\documentclass[11pt]{article}

\usepackage{amsfonts}
\usepackage{amsmath}
\usepackage{amsthm}
\usepackage{graphicx}
\usepackage{hyperref}
\usepackage{color}
\usepackage{soul}

\newtheorem{theorem}{Theorem}

\newtheorem{corollary}[theorem]{Corollary}

\newtheorem{lemma}[theorem]{Lemma}

\newtheorem{proposition}[theorem]{Proposition}

\def\SIGMA{{s}}

\def\ucr{{\text{\rm cr}}}

\def\ignore#1{{}}

\begin{document}

\title{Shellable drawings and the cylindrical crossing number of $K_{n}$}
\author{Bernardo M.~\'{A}brego \\
%EndAName
{\small California State University, Northridge}\\
{\small bernardo.abrego@csun.edu } \and  Oswin Aichholzer  \\
%EndAName
{\small Graz University of Technology}\\
{\small oaich@ist.tugraz.at} \and Silvia Fern\'{a}ndez-Merchant \\
%EndAName
{\small California State University, Northridge}\\
{\small silvia.fernandez@csun.edu} \and Pedro Ramos \\
%EndAName
{\small Universidad de Alcal\'{a}}\\
{\small pedro.ramos@uah.es} \and Gelasio Salazar \\
%EndAName
{\small Universidad Aut\'{o}noma de San Luis Potos\'{\i}}\\
{\small gsalazar@ifisica.uaslp.mx}}
\maketitle

\begin{abstract}
The Harary-Hill Conjecture states that the number of crossings in any drawing of the complete graph $ K_n $ in the plane is at least
  $Z(n):=\frac{1}{4}\left\lfloor \frac{n}{2}\right\rfloor \left\lfloor
    \frac{n-1}{2}\right\rfloor \left\lfloor \frac{n-2}{2}\right\rfloor
  \left\lfloor \frac{n-3}{2}\right\rfloor$. In this paper, we settle the Harary-Hill conjecture for {\em shellable drawings}. We say that a drawing $D$ of $ K_n $ is {\em $ s $-shellable} if there exist a subset $ S = \{v_1,v_2,\ldots ,v_ s\}$ of the vertices and a region $R$ of $D$ with the following property: For all $1 \leq i < j \leq s$, if $D_{ij}$ is  the drawing obtained from $D$ by removing $v_1,v_2,\ldots v_{i-1},v_{j+1},\ldots,v_{s}$, then $v_i$ and $v_j$ are on the boundary of the region of $D_{ij}$ that contains $R$. For $ s\geq n/2 $, we prove that the number of crossings of any $ s $-shellable drawing of $ K_n $ is at least the long-conjectured value $Z(n)$. Furthermore, we prove that all cylindrical, $ x $-bounded, monotone, and 2-page drawings of $ K_n $ are $ s $-shellable for some $ s\geq n/2 $ and thus they all have at least $ Z(n) $ crossings. The techniques developed provide a unified proof of the Harary-Hill conjecture for these classes of drawings.
\end{abstract}

\section{Introduction}\label{sec:introduction}

In the late 1950s, the British artist Anthony Hill got interested in producing drawings of the complete graph $K_n$ with the least possible number of edge crossings. His general technique, explained in a paper he wrote jointly with Harary~\cite{HH63}, is best described by drawing $K_n$ on a cylinder as follows. Draw a cycle with $\lceil{n/2}\rceil$ vertices on the rim of the top lid, and a cycle with the remaining $\lfloor{n/2}\rfloor$ vertices on the rim of the bottom lid. Then draw the remaining edges joining vertices on the same lid using the straight line joining them across the lid. Finally, for any two vertices on distinct lids, draw the edge joining them along the geodesic that connects them on the side of the cylinder. (See Figure \ref{fig:FigCylindrical_2Page}, left, for a planar representation of such a drawing.) It is an elementary exercise to show that such a drawing of $K_n$ has exactly
$Z\left( n\right) :=\frac{1}{4}\left\lfloor \frac{n}{2}\right\rfloor
\left\lfloor \frac{n-1}{2}\right\rfloor \left\lfloor
\frac{n-2}{2}\right\rfloor \left\lfloor \frac{n-3}{2}\right\rfloor$
crossings. The Harary-Hill constructions are a particular instance of {\em
  cylindrical} drawings (see formal definition in Section~\ref{sec:verifyinghh}).

At about the same time as the Harary-Hill paper was published,
Bla\v{z}ek and Koman got independently interested in drawing $K_n$
with as few crossings as possible~\cite{BC64}. In their construction (see Figure \ref{fig:FigCylindrical_2Page}, right),
they start by drawing a cycle as a regular $n$-gon, and then drawing all diagonals with positive slope (as straight line segments) and all other edges outside the cycle.
%\marginpar{\tiny Check: these guys use pages or work on the sphere or what?}
The Bla\v{z}ek-Koman construction also yields drawings of $K_n$ with
exactly $Z(n)$ crossings, and it is a particular instance of {\em
  2-page} drawings (see below for the definition).

\begin{figure}[h]
\centering
\includegraphics[width=0.6\linewidth]{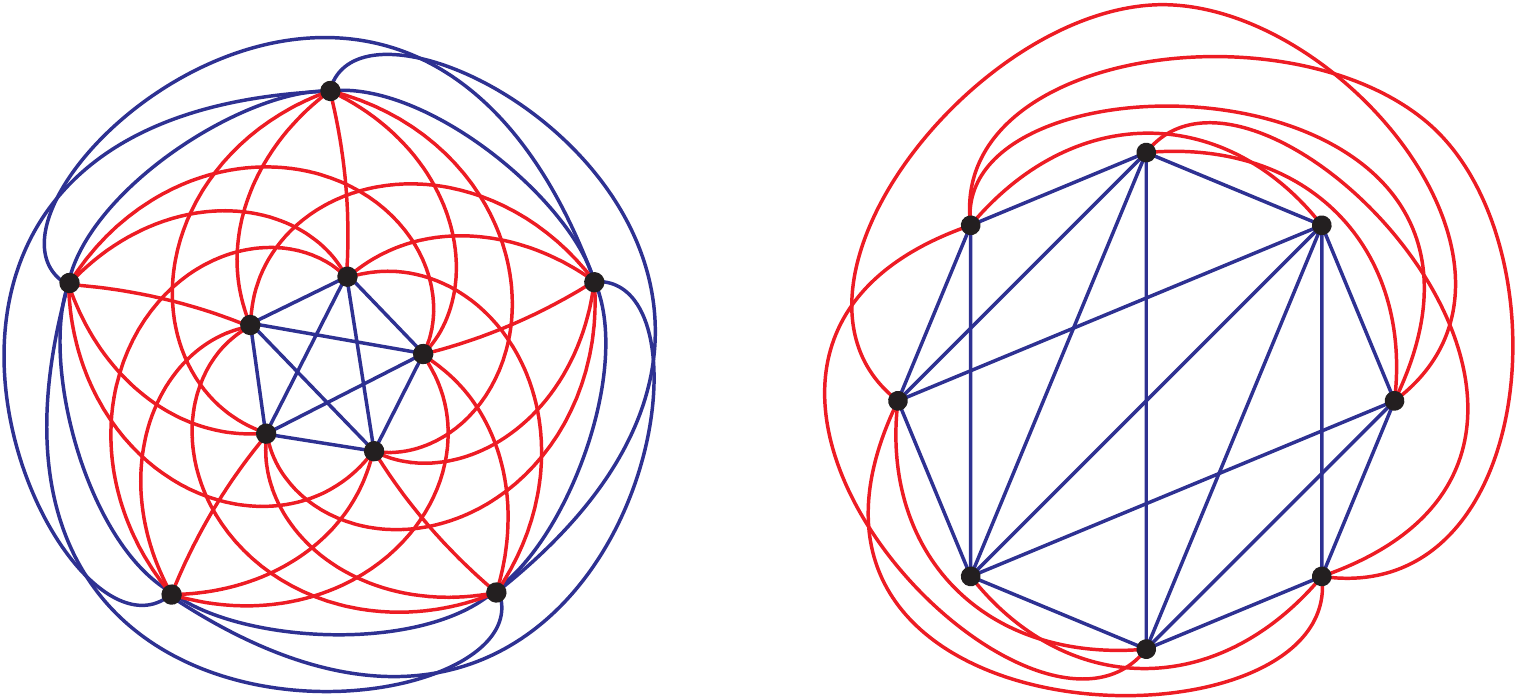}
\caption{\emph{Left}: Harary-Hill construction for 10 points. (A cylindrical drawing.) \emph{Right}: Bla\v{z}ek-Koman construction for 8 points. (A 2-page drawing.)}
\label{fig:FigCylindrical_2Page}
\end{figure}

To this date, no drawing of $K_n$ with fewer than $Z(n)$ crossings is
known. Moreover, all general constructions (for arbitrary values of
$n$) known with exactly $Z(n)$ crossings are
obtained from insubstantial alterations of either the Harary-Hill or
the Bla\v{z}ek-Koman constructions (a few exceptions are known, but
only for some small values of $n$).
{The tantalizingly open} {\em Harary-Hill conjecture}
$\ucr(K_n)=Z(n)$ has been
confirmed only for $n \le 12$~\cite{PR}.

The main contribution of this paper is the introduction of {\em
  shellable} drawings, a large class of drawings for which (as we shall
show) the Harary-Hill conjecture holds. Shellability captures the
essential features of 2-page drawings we previously used~\cite{AAFRS2,AAFRS} to prove
that the 2-page crossing number of $K_n$ is $Z(n)$, and allows us to
  extend the lower bound to a larger family of drawings, including
  cylindrical, monotone, and $x$-bounded drawings (see definitions
  below).

If a drawing $D$ of a graph is regarded as a subset of the plane, then a {\em region} of $D$ is a connected component of $\mathbb{R}^2\setminus D$. (If $D$ is an
embedding, then the regions of $D$ are the faces).  {A drawing $D$ of
$K_n$ is $s$-{\em shellable} if there exists a subset $ S = \{v_1,v_2,\ldots ,v_ s\}$ of the
vertices and a region $R$ of $D$ with the following
property. For $1\leq i < j \leq s$, if $D_{ij}$ denotes the drawing obtained
from $D$ by removing $v_1,v_2,\ldots v_{i-1},
v_{j+1},v_{j+2},\ldots,v_{s}$, then for all $1 \leq i < j \leq s$, the
vertices $v_i$ and
$v_j$ are on the boundary of the region of $D_{ij}$ that contains $R$.} The set $ S $ is an \emph{$ s $-shelling} of $D$ {\emph{witnessed} by $ R $.

The core of this paper is the following statement, whose proof is
given in Section~\ref{sec:pro1}.

\begin{theorem}\label{th:HHshellable}
Let $ D $ be an $ s $-shellable drawing of $ K_n $, for some $ s\geq
n/2$. Then $ D$ has at least $Z(n)$ crossings.
\end{theorem}

We use this to settle the Harary-Hill conjecture for several classes of drawings:

\begin{itemize}
\item In a {\em 2-page book drawing} (or simply {\em 2-page drawing}), the vertices are placed on a
  line (the {\em spine} of the {\em book}), and each edge (except for
  its endvertices) lies entirely on {an open halfplane spanned by the spine (one of the 2 {\em pages} of the book)}. (See Figure \ref{fig:FigGeneralCylindrical_2Page}, right.)
\item Following Schaefer~\cite{SC}, in a {\em cylindrical drawing} of
  a graph, there are two concentric circles that host all the vertices,
  and no edge is allowed to intersect these circles, other than at its
  endvertices. (Schaefer defines cylindrical drawings only for
  bipartite graphs, but his definition obviously applies to arbitrary
  graphs). (See Figure \ref{fig:FigGeneralCylindrical_2Page}, left.)
\end{itemize}

\begin{figure}[h]
\centering
\includegraphics[width=0.6\linewidth]{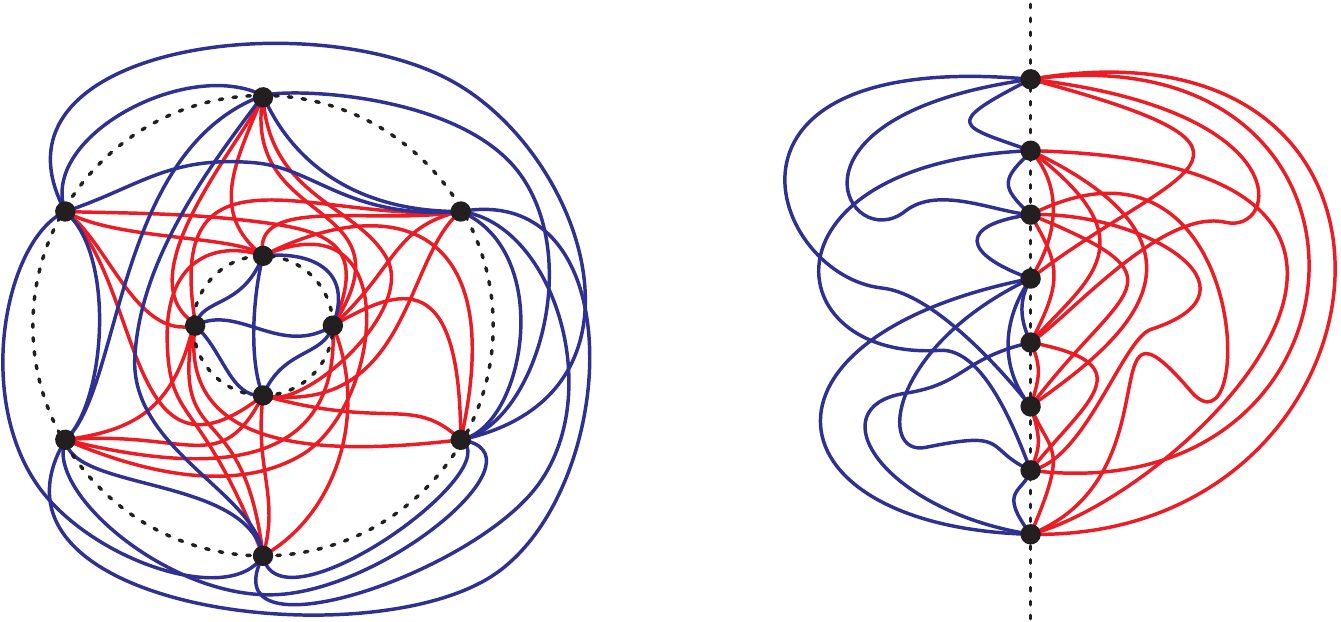}
\caption{\emph{Left:} A cylindrical drawing of $K_{10}$. \emph{Right:} A 2-page drawing of $K_{8}$.}
\label{fig:FigGeneralCylindrical_2Page}
\end{figure}

We remark that Hill's drawings can be naturally regarded as
cylindrical drawings. Indeed, even though
in Hill's drawings the edges joining consecutive rim vertices are
placed on the rims, such drawings are easily adapted to this
definition, since those edges can be drawn arbitrarily close to a rim.

\begin{itemize}
\item A drawing is {\em monotone} if each vertical line
intersects each edge at most once. (See Figure \ref{fig:FigMonotone_Bounded}, right.)

\item A drawing is {\em $x$-bounded} if
by labelling the vertices $v_1, v_2, \ldots, v_n$ in increasing order
of their $x$-coordinates, for all $1 \le i < j \le n$ the edge $v_i
v_j$ is contained in the strip bounded by the vertical line that contains
$v_i$ and the vertical line that contains $v_j$. (See Figure \ref{fig:FigMonotone_Bounded}, left.)
\end{itemize}

\begin{figure}[h]
\centering
\includegraphics[width=0.6\linewidth]{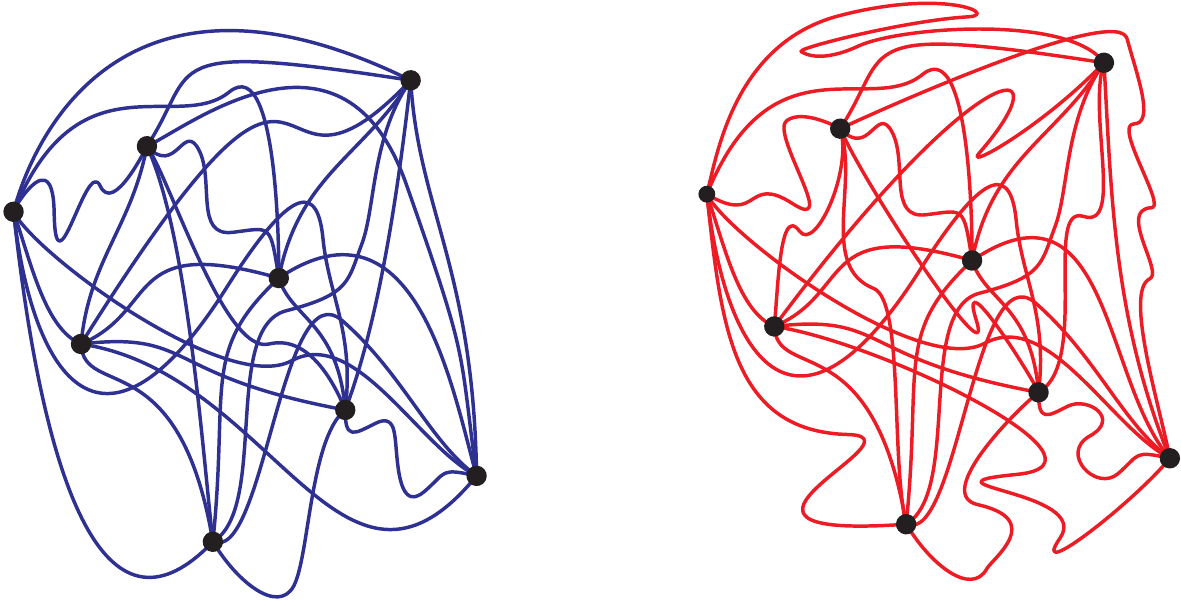}
\caption{\emph{Left:} A monotone drawing of $ K_8 $. \emph{Right:} An x-bounded drawing of $ K_8 $.}
\label{fig:FigMonotone_Bounded}
\end{figure}

In Section~\ref{sec:verifyinghh}, we find a
  condition on drawings of $K_n$ that guarantees that they are
  $s$-shellable for some $s\ge n/2$. Then we show that if $D$ is
  a crossing minimal
  2-page, cylindrical, monotone, or $x$-bounded drawing, then
  $D$ satisfies
  this condition, thus settling (in view of Theorem~\ref{th:HHshellable}) the Harary-Hill conjecture for all
  these families of drawings.
  Section~\ref{sec:cr} contains some concluding remarks.

\section{$k$-edges in shellable drawings and proof of Theorem~\ref{th:HHshellable}}\label{sec:pro1}

We recall that in a {\em good} drawing of a graph, no two edges share more than one point and no edge crosses itself. It is easy to show that every crossing minimal drawing of a
graph is good.

We generalized the geometrical concept of a $k$-edge
to arbitrary  {(topological)}  good drawings of $K_n$~\cite{AAFRS2,AAFRS}, as follows. Let $D$ be a good drawing of $K_n$, $pq$ a directed edge of $D$, and $r$ a vertex of $D$ distinct from $p$ and $q$. Then $pqr$ denotes the oriented closed curve defined by concatenating the edges $pq$, $qr$, and $rp$. An oriented, simple, and closed curve in the plane is oriented {\em counterclockwise} (respectively, {\em clockwise}) if the bounded region it encloses is on the left (respectively, right) hand side of the curve.  {Further,} $r$ is {\em on the left} (respectively, {\em right}) {\em side of} $pq$ if $pqr$ is oriented counterclockwise (respectively, clockwise). We say that the edge $pq$ is a $k$-edge of $D$ if it has exactly $k$ points of $D$ on one side (left or right), and thus $n-2-k$ points on the other side. Hence, as in the geometric setting, a $k$-edge is also an $(n-2-k)$-edge. The direction of the edge $pq$ is no longer relevant and every edge of $D$ is a $k$-edge for some unique $k$ such that $0 \le k \le \lfloor{n/2}\rfloor-1$.

Following our previous work~\cite{AAFRS2,AAFRS}, if $D$ is a good drawing of $K_n$, then for each $0\leq k\leq\lfloor n/2\rfloor-1$ we define the set of ${\leq}k$\emph{-edges} of $D$ as all $j$-edges in $D$ for $j=0,\ldots,k$. The number of ${\leq}k$-edges of $D$ is denoted by
\[
E_{{\leq}k}\left(  D\right)  :=\sum\limits_{j=0}^{k}E_{j}\left(  D\right)  .
\]
Similarly, we denote the number of ${\leq}{\leq}k$\emph{-edges }of $D$ by
\begin{equation}\label{e:atmostatmost}
E_{{\leq}{\leq}k}\left(  D\right)  :=\sum\limits_{j=0}^{k}E_{\leq j}\left(
D\right)  =\sum\limits_{j=0}^{k}\sum\limits_{i=0}^{j}E_{i}\left(  D\right)
=\sum\limits_{i=0}^{k}\left(  k+1-i\right)  E_{i}\left(  D\right)  .
\end{equation}
It is known \cite{AAFRS2,AAFRS} that if $ D $ is a good
drawing, then $D$ has exactly
\begin{equation}\label{e:crossingsvsedges}
2\sum_{k=0}^{\lfloor{n/2}\rfloor-2} E_{\le \le k}(D) -\frac{1}{2}{n\choose 2}\biggl\lfloor{\frac{n-2}{2}}\biggr\rfloor -\frac{1}{2} \left(1 + (-1)^n  \right) E_{\le \le \lfloor{n/2}\rfloor-2}(D)
\end{equation}
crossings. Thus we now concentrate on bounding $ E_{\le \le k}(D) $. We need a few more definitions. If $ D_y $ is the drawing of $ K_{n-1} $ obtained from $ D $ by deleting a vertex $ y $, then an edge non-incident to $ y $  is $ (D,D_y) $-{\em invariant} if for some $ 0\leq k \leq \left\lfloor(n-3)/2\right\rfloor $ it is a $ k $-edge in both $ D $ and $ D_y $. We let $E_{\leq k}(D,D_y)$ denote the number of $ (D,D_y) $-invariant $\leq k$-edges.

%\section{Cylindrical and shellable drawings}
%
%\begin{lemma}\label{lem:thereexist}
%For each integer $n \ge 3$, there exists a crossing-minimal cylindrical drawing of $K_n$ that is $s$-shellable for some $s \geq n/2$.
%\end{lemma}

\subsection{Ordering the vertices with respect to a boundary point}

{The {\em unbounded region} of a drawing $D$ is its unique region with
noncompact closure. We refer to the  topological boundary of the
unbounded region of~$D$ simply as the {\em boundary of~$D$}.}

Let $ D $ be a good drawing of $ K_n $ and assume that $ x $  is a vertex on the boundary of $ D $. Then there is a natural order of the vertices of $ D_x $ induced by the order in which the edges of $ D $ \emph{leave}~$ x $. Namely, there is a disk $ \Omega $ with center $ x $ and radius $ \epsilon >0$ that intersects $ D $ only at the edges incident to $ x $. Moreover, for $ \epsilon $ small enough, $ \Omega $ intersects each edge incident to $ x $ in a simple connected Jordan curve. (See Figure \ref{fig:FigBoundaryOrder}.) Exactly two of these curves, say $ xy \cap \Omega$ and $ xz \cap \Omega$ for some vertices $ y $ and $ z $,  are on the boundary of $ D $. Suppose without loss of generality that the triangle $ xyz $ is oriented counter-clockwise. Then we can label the vertices of $ D_x $ by $ x_1,x_2,\dots,x_{n-1} $ so that $ x_1=y $, $ x_{n-1}=z $, and the Jordan curves $ xx_1\cap \Omega, xx_2\cap \Omega, \ldots ,xx_{n-1}\cap \Omega $ appear in counter-clockwise order around $ x $. We refer to this  as \emph{the order induced by $ x $ in $ D $}.

\begin{figure}
\centering
\includegraphics[width=0.5\linewidth]{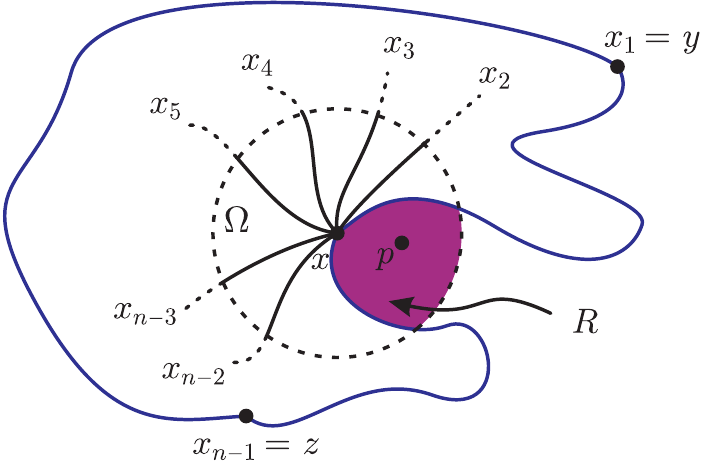}
\caption{The order induced by $x$.}
\label{fig:FigBoundaryOrder}
\end{figure}

\begin{proposition}\label{pro:lastpoint}
Let $ n\geq 1 $ and consider a good drawing $ D $ of the complete
graph $ K_n $. Let $ x $ be a vertex on the boundary of $ D $,  and
let $ x_1,x_2, \ldots,x_{n-1} $ be the order induced by $ x $ in $ D $. Then $ xx_i $ and $ xx_{n-i} $ are $i-1$-edges of $ D $ for $ 1\leq i \leq \left\lfloor(n-2)/2\right\rfloor$.
\end{proposition}

\begin{proof}
Consider a disk $ \Omega $ as above. Then any point $ p $ in $\Omega $ and outside the triangle $ xyz $ is in the unbounded region of $ D $. (See Figure \ref{fig:FigBoundaryOrder}.) This means that $ p $ cannot be in the interior of any triangle of $ D $. In particular, if $ j<i $, then the triangle $ xx_jx_i $ is oriented counter-clockwise as otherwise its interior would contain $ p $. This means that $ x_j $ is to the right of $ xx_i $ if $ j < i $, and to the left if {$ j>i $}. Thus there are exactly $ i-1 $ vertices to the right of $ xx_i$ and $ n-1-i $ to the left. This means that $ xx_i $ is a $ \min(i-1,n-1-i) $-edge of $ D $, implying the result.
\end{proof}

\begin{proposition}\label{pro:invariant}
Let $ 0\leq i-1 \leq k \leq \left\lfloor(n-3)/2\right\rfloor $, $ D $ a good drawing of the complete graph $ K_n $, and $ x $ and $ y $ vertices of $ D $. Let $ U $ be a subset of $ i-1 $ vertices of $ D $ not including $ x $ and $ y $. Assume that $ x $ is on the boundary of the drawing $ D(U) $ obtained from $ D $ by removing $ U $. Then there exist at least $ k-i+2 $ edges incident to $x$ and non-incident to vertices in $ U $ that are $ (D,D_y) $-invariant $\leq k$-edges.
\end{proposition}

\begin{proof}
Consider the order $ x_1,x_2,\ldots x_{n-i} $ induced by $ x $ in $ D(U) $. As before, $ x_\ell $ is to the right of $ xx_j $ if $ \ell < j $, and to the left if $ \ell >j $. Thus there are exactly $ j-1 $ vertices in $D(U)$ to the right of $ xx_j $ and $ n-i-j $ to the left. Including $ U $, this means that there are at most $ i-1+j-1=i+j-2 $ vertices to the right of $ xx_j $ in $ D $ and at most $ i-1+n-i-j=n-j-1 $ to the left.

Now consider the point $ y $, which is equal to $ x_w $ for some $ 1\leq w \leq n-i $. If $ w>k+2-i $, then for $ 1\leq j \leq k+2-i $ the edge $ xx_j $ has at most $ i+j-2 \leq i+(k+2-i)-2 = k$ points to its right and $ y $ on its left (because $ w>k+2-i \geq j $). If $ w \leq k+2-i $, then for $ n-k-1 \leq j \leq n-i $ the edge $ xx_j $ has at most $ n-j-1 \leq n-(n-k-1)-1 =k $ points to its left and $ y $ on its right (because $ k \leq (n-3)/2 < (n-3+i)/2 $ and thus $ w\leq k+2-i < n-k-1 \leq j $). In either case, the $ k+2-i $ edges $ xx_j $ are $ (D,D_y) $-invariant $\leq k$-edges.
\end{proof}

\subsection{Bounding the number of $\leq \leq k$-edges in shellable drawings of $ K_n $}

We now bound the number of $\leq\leq k$-edges of $ s $-shellable drawings of $ K_n $ for a certain interval of $ k $ determined by $ s $.

\begin{proposition}\label{pro:use}
Let $ D $ be an $ s $-shellable good drawing of the complete graph $
K_n $, {in which the region $R$ that witnesses the $s$-shellability of
$D$ is its unbounded region.} Then $E_{\leq \leq k} (D) \geq 3
\binom{k+3}{3}$ for all $ 0 \leq k \leq \min (s-2,\left\lfloor(n-3)/2\right\rfloor) $.
\end{proposition}

\begin{proof}
Let $ V $ be the set of vertices of $ D $ and $ S = \{v_1,v_2,\ldots ,v_ s\}$ an $ s $-shelling of $ D $ {witnessed by the unbounded region $ R $}. Fix $ k $ with $ 0 \leq k \leq \min (s-2,\left\lfloor(n-3)/2\right\rfloor) $. We prove that
\begin{equation}\label{e:subdrawings}
E_{{\leq }{\leq i}}(D_{1,s-k+i})\geq 3 \binom{i+3}{3}
\end{equation}
for $ 0 \leq i \leq k $ by induction on $ i $. For $ i=0 $, because $ S $ is an $ s $-shelling of $ D $, {and the unbounded region
witnesses this $s$-shellability,} it follows that $ v_1$ and $v_{s-k} $ are on the boundary of $ D_{1,s-k} $. By Proposition~\ref{pro:lastpoint} each of these two vertices (they are different because $ k \leq s-2 $) is incident to two 0-edges and they can share at most one 0-edge. That is, $ E_{{\leq}{\leq}0}(D_{1,s-k}) \geq 3 $.
We now compare the following two identities obtained from (\ref{e:atmostatmost}). For $ 1\leq r \leq s $ and $ 0 \leq k' \leq \left\lfloor(n-s+r)/2\right\rfloor $,
\begin{equation}\label{e:allvertices}
E_{{\leq}{\leq}k'}(D_{1,r})=\sum_{j=0}^{k'}(k'+1-j)E_j(D_{1,r})
\end{equation}
and
\begin{equation}\label{e:allverticesbutone}
E_{{\leq}{\leq}k'-1}(D_{1,r-1})=\sum_{j=0}^{k'-1}(k'-j)E_j(D_{1,r-1}).
\end{equation}

As shown in our previous work \cite{AAFRS1}, for a $ j \leq k' $ a $ j $-edge incident to $ v_r $ contributes $ k'-j $ to (\ref{e:allvertices}) and nothing to (\ref{e:allverticesbutone}), a $ (D_{1,r},D_{1,r-1}) $-invariant edge contributes 1 more to (\ref{e:allvertices}) than to (\ref{e:allverticesbutone}), and all other edges contribute the same to (\ref{e:allvertices}) and (\ref{e:allverticesbutone}). Therefore,

\begin{equation}
E_{{\leq}{\leq}k'}(D_{1,r})=E_{{\leq}{\leq}k'-1}(D_{1,r-1})+\sum_{\ell=0}^{k'}(k'+1-\ell)e_\ell(v_r)+E_{\leq k}(D_{1,r},D_{1,r-1}),\label{NewIdentity}
\end{equation}
where $ e_\ell(r) $ is the number of $ \ell $-edges incident to $ v_r $ in $ D_{1,r}$.

%For $ i=1 $, (\ref{NewIdentity}) for $ k'=1 $ and $ r=s-k+1 $ states that
%\begin{equation}
%E_{{\leq}{\leq}1}(D_{1,s-k+1})=E_{{\leq}{\leq}0}(D_{1,s-k})+\sum_{\ell=0}^{1}(2-\ell)e_\ell(v_{s-k+1})+E_{\leq
%k}(D_{1,s-k+1},D_{1,s-k}).\label{e:Basis}
%\end{equation}
% Now,  $v_{s-k+1} $ is on the boundary of $ D_{1,s-k+1} $ and thus by Proposition~\ref{pro:lastpoint} we have that $ e_\ell (v_{s-k+1})=2 $. Hence, $ \sum_{\ell=0}^{1}(2-\ell)e_\ell(v_{s-k+1})=6.$ Finally, by Proposition~\ref{pro:invariant}, there are at least two $ (D_{1,s-k+1},D_{1,s-k}) $-invariant $ \leq 1 $-edges incident to $ v_1 $ (for the drawing $ D_{1,s-k+1} $, $ i=k=1 $, $ x=v_1 $, and $ U=\emptyset $) and at least one incident to $ v_2 $ but not to $ v_1 $ (for the drawing $ D_{1,s-k+1} $, $ i=2$, $k=1 $, $ x=v_2 $, and $ U=\{v_1\}$). That is, $ E_{\leq
%k}(D_{1,s-k+1},D_{1,s-k}) \geq 3 $. Therefore, the right hand side of (\ref{e:Basis}) is at least $ 12=3 \binom{1+3}{3} $.

Now, choose $ i $ such that  $ 1 \leq i \leq k$ and assume that
\begin{equation}\label{e:indhyp}
E_{{\leq }{\leq i-1}}(D_{1,s-k+i-1})\geq 3 \binom{i+2}{3}.
\end{equation}
By (\ref{NewIdentity}) for $ k'=i $ and $ r=s-k+i $, we have that
\begin{equation}
E_{{\leq}{\leq}i}(D_{1,s-k+i})=E_{{\leq}{\leq}i-1}(D_{1,s-k+i-1})+\sum_{\ell=0}^{i}(i+1-\ell)e_\ell(v_{s-k+i})+E_{\leq i}(D_{1,s-k+i},D_{1,s-k+i-1}),\label{e:InductionStep}
\end{equation}
We separately bound each term of the right-hand side of (\ref{e:InductionStep}). The first term is bounded in (\ref{e:indhyp}). For the second term, Proposition~\ref{pro:lastpoint} (for $ x=v_{s-k+i} $ is on the boundary of $ D_{1,s-k+i} $) implies that $ e_\ell(v_{s-k+i}) =2 $ and thus
\begin{equation}
\sum_{\ell=0}^{i}(i+1-\ell)e_\ell(v_{s-k+i})=\sum_{\ell=0}^{i}(i+1-\ell)2=2\binom{i+2}{2}.
\end{equation}
Finally, we show that
\begin{equation}
E_{\leq i}(D_{1,s-k+i},D_{1,s-k+i-1}) \geq  \sum_{\ell=1}^{i+1}(i-\ell+2) = \binom{i+2}{2}.
\end{equation}
We use Proposition~\ref{pro:invariant} for the drawing $ D_{\ell,s-k+i} $, $ x=v_\ell,y=v_{s-k+i} $, and $ U=\{v_1,v_2\ldots , v_{\ell -1}\} $. Note that $ k\leq s-2 $ implies $ 1 \leq \ell \leq i+1 < s-k+i $ and thus $ v_\ell $ and $ v_{s-k+i} $ are different and do not belong to $ \{v_1,v_2,\ldots v_{\ell -1}\} $. Moreover, $ v_\ell $ and $ v_{s-k+i} $ are on the boundary of $ D_{1,s-k+i} $ because $ S $ is an $ s $-shelling of $ D $. Also, $ D_{\ell,s-k+i} $ has $ n-s+(s-k+i)=n-k+i $ vertices and thus we must check that $ 0 \leq \ell -1 \leq i \leq (n-k+i-3)/2 $. The first two inequalities hold because $ 1 \leq \ell \leq i+1 $. The last inequality follows from $ k \leq \min(s-2,\left\lfloor(n-3)/2\right\rfloor)\leq \left\lfloor(n-3)/2\right\rfloor $, which implies $ k+i \leq 2k \leq n-3 $. Therefore, Proposition \ref{pro:invariant} implies that for $1 \leq \ell \leq i+1$ there are at least $i-\ell+2$ edges incident to $ v_\ell $ and non-incident to $ v_1,v_2,\ldots ,v_{\ell -1} $ (so all these edges are different) that are $ (D_{s-k+i-1},D_{s-k+i}) $-invariant $\leq i$-edges.
\end{proof}

\subsection{Proof of Theorem~\ref{th:HHshellable}}

Let $D$ be an $\SIGMA$-shellable drawing of $K_n$, for some $s \geq
n/2$. {By using a suitable inversion, if needed, we transform $D$ into
a drawing $D'$, with the same number of crossings as $D$, such that the
region that witnesses the $s$-shellability of $D'$ is the unbounded region.}
Since
$\min(s-2,\left\lfloor(n-3)/2\right\rfloor)=\left\lfloor(n-3)/2\right\rfloor$, it follows from
Proposition~\ref{pro:use} that $E_{\leq \leq k} (D') \geq 3 \binom{k+3}{3}$ for all $ 0 \leq k \leq \left\lfloor(n-3)/2\right\rfloor$.

Since $D'$ is a good
drawing, then by (\ref{e:crossingsvsedges}) $D'$ has exactly
\[
2\sum_{k=0}^{\lfloor{n/2}\rfloor-2} E_{\le \le k}(D') -\frac{1}{2}{n\choose 2}\biggl\lfloor{\frac{n-2}{2}}\biggr\rfloor -\frac{1}{2} \left(1 + (-1)^n  \right) E_{\le \le \lfloor{n/2}\rfloor-2}(D')
\]
crossings. Using this fact, a straightforward calculation \cite{AAFRS2,AAFRS} shows that if $D'$ is a
drawing of $K_n$ that satisfies $E_{\leq \leq k} (D') \geq 3
\binom{k+3}{3}$ for all $ 0 \leq k \leq
\left\lfloor(n-3)/2\right\rfloor $, then $D'$ has at least $Z(n)$
crossings.\hfill $\Box$

\section{Verifying the Harary-Hill conjecture for 2-page,
  cylindrical, \\ monotone, and $x$-bounded drawings}\label{sec:verifyinghh}

The workhorse of this section is a property of a drawing that
guarantees its shellability:

\begin{lemma}\label{lem:generalized}% (Generalized version)
Let $D$ be a drawing of $K_n$. Suppose that $C=v_1 v_2 \ldots v_s$ is
a cycle that satisfies the following: (i) the edge
$v_sv_1$ has no crossings; and (ii) for
$k=1,\ldots,s-1$ all crossings in the edge $v_kv_{k+1}$ involve edges
$v_iv_j$ with $i<k$ and $j>k+1$. Then $D$ is  $\SIGMA$-shellable.
\end{lemma}

\begin{proof}
Let $R$ be a region of $ D $ containing the edge $ v_sv_1 $ on its boundary.  Let $1 \le i < j \le s$ and define $D_{ij}$ as before. Let $ R' $ be the region of $ D_{ij} $ that contains $ R $. Since the vertices $v_1, v_2, \ldots,
v_{i-1}, v_{j+1}, v_{j+2},\ldots, v_s$, and consequently any edge incident to one of these vertices, are
removed to obtain $D_{ij}$, then $ v_1 $ and $ v_s $ are in the interior of $ R' $. Moreover, it follows from the crossing properties of the edges of $C$ that the edges $v_1 v_2, v_2
v_3, \ldots, v_{i-1} v_i, v_{j}v_{j+1},$ $v_{j+1}v_{j+2}, \ldots, v_{s-1}v_s$ are not intersected by any edge of $ D_{ij} $. Hence the paths $ v_i,v_{i-1},\ldots ,v_1 $ and $ v_j, v_{j+1},\ldots ,v_s$ are completely contained in $ R' $ and thus $ v_i $ and $ v_j $ are on the boundary of $ R $. Therefore, $\{v_1, v_2, \ldots, v_s\}$ is an $s$-shelling of $D$ witnessed by $ R $.
\end{proof}

We need the full strength of Lemma~\ref{lem:generalized} to show that
monotone and $x$-bounded drawings satisfy the Harary-Hill
conjecture. However, it seems worth stating the following
weaker form, which is all we need to show that the
Harary-Hill conjecture holds for 2-page and cylindrical drawings:

\begin{corollary}\label{cor:crfree}
If a drawing $D$ of $K_n$ has a crossing-free cycle $C$ of size $s$
then $D$ is  $\SIGMA$-shellable.\hfill $\Box$
\end{corollary}

{We are finally ready to verify the Harary-Hill conjecture for several
classes of drawings.}

\begin{theorem}\label{thm:cyl}
Every cylindrical drawing of $K_n$ has at least $Z(n)$ crossings.
\end{theorem}

\begin{proof}
%\ignore{{The Harary-Hill construction shows that there are
%cylindrical drawings with exactly $Z(n)$ crossings, and so we need
%to prove the lower bound.}}
{Let $D$ be a crossing-minimal cylindrical drawing of $K_n$. Out
  of the two concentric cycles that contain all the vertices, let
  $\rho$ be one that contains at least $n/2$ vertices.}
{Let $v_1, v_2, \ldots, v_s$ be the
vertices on $\rho$, in counterclockwise order. Since no two edges
cross each other more than once (this follows since $D$
is crossing-minimal)
and no edge crosses $\rho$, it follows that the cycle  $v_1 v_2 \ldots
v_s v_1$ is uncrossed in $D$. Since $s\ge n/2$, the result follows by
Theorem~\ref{th:HHshellable} and Corollary~\ref{cor:crfree}.}
\end{proof}

A 2-page drawing is a particular kind of a cylindrical drawing, namely, a degenerate one with all vertices on one of the
  concentric circles. Thus Theorem~\ref{thm:cyl} immediately implies our previous result ~\cite{AAFRS2,AAFRS} for 2-page drawings:

\begin{corollary}\label{cor:2page}
Every 2-page drawing of  $K_n$ has at least $Z(n)$ crossings.\hfill $\Box$
\end{corollary}

{It is straightforward to check that any $x$-bounded drawing $D$ of $K_n$
satisfies the conditions of
Lemma~\ref{lem:generalized}. Thus the Harary-Hill conjecture holds for $x$-bounded drawings:}

\begin{theorem}\label{thm:xbounded}
Every $x$-bounded drawing of $K_n$ has at least $Z(n)$ crossings. \hfill $\Box$
\end{theorem}

Since every monotone drawing is obviously $x$-bounded, this implies the Harary-Hill conjecture for monotone drawings (previously proved by the authors~\cite{AAFRS1} and by Balko et al.~\cite{kyncl}):

\begin{corollary}\label{thm:monotone}
Every monotone drawing of $K_n$ has at least $Z(n)$ crossings.\hfill $\Box$
\end{corollary}

%Extending the ideas of that proof, we show below the following:
%
%\begin{lemma}\label{lem:HHshellable}
%Suppose that $ D $ is a $\sigma$-shellable drawing of $ K_n $, for some $ s\geq  n/2$. Then $ D$ has at least $Z(n)$ crossings.
%\end{lemma}
%
%It is easy to see that not every cylindrical drawing is $\sigma$-shellable for some $s$ to which this last statement can be applied. On the other hand, we have shown the existence of crossing-minimal cylindrical drawings that satisfy this shellability condition:
%
%\begin{lemma}\label{lem:thereexist}
%For each integer $n \ge 3$, there exists a crossing-minimal cylindrical drawing of $K_n$ that is $\sigma$-shellable for some $s \geq n/2$.
%\end{lemma}
%
%Lemmas~\ref{lem:HHshellable} and~\ref{lem:thereexist} combine to prove that every cylindrical drawing of $K_n$ has at least $Z(n)$ crossings. Since there exist cylindrical drawings of $K_n$ with exactly $Z(n)$ crossings (such as the original Harary-Hill drawings~\cite{HH63}), Theorem~\ref{thm:cyl} follows.
%
%The proofs of Lemmas~\ref{lem:HHshellable} and~\ref{lem:thereexist} are given in Sections~\ref{sec:pro1} and~\ref{sec:pro2}, respectively. Section~\ref{sec:cr} contains some concluding remarks and open questions.

\section{Concluding remarks}\label{sec:cr}

Cylindrical drawings of $K_n$ were previously investigated
by Richter and Thomassen~\cite{RT}. In that paper, they determined the
number of crossings in a cylindrical drawing of $K_{m,m}$ with one
chromatic class on the inner circle and the other chromatic class on
the outer circle. From their result it follows that a cylindrical
drawing of $K_{2m}$ in which the edges joining vertices on the same
circle are {not} drawn on the annulus (bounded by the two circles) has
at least $Z(2m)$ crossings.

As we observed in Section~\ref{sec:introduction}, the 2-page and the
cylindrical constructions (possibly with some insubstantial
alterations) are the only known drawings of $K_n$ with $Z(n)$ crossings for arbitrary values of $n$. In his
interesting entry at {\tt mathoverflow.net}, Kyn\v{c}l~\cite{kynclmo} asks about the
existence of alternative constructions, and observes that there is a
plethora of drawings with $Z(n) + O(n^3)$ crossings (noting that Moon
showed that a random spherical drawing of $K_n$ has expected crossing
number $(1/64)n(n-1)(n-2)(n-3)=Z(n)+O(n^3)$).

Balko et~al.~\cite{kyncl} noted that there are cylindrical drawings $ D $ that do not satisfy  the bound $E_{\leq \leq k} (D) \geq 3 \binom{k+3}{3}$. However, as shown in this paper, for every such drawing there exists a second drawing $ D' $ obtained from $ D $ by an appropriate inversion (and thus with the same number of crossings) that satisfies $E_{\leq \leq k} (D') \geq 3 \binom{k+3}{3}$.

%%% Any hope of ``characterizing'' optimal cylindrical drawings? Oswin
%%% has found optimal cylindrical drawings of $K_9$ with $6$ vertices
%%% on one lid and $3$ on the other. So\ldots it not true that ``half
%%% and half'' always occurs in optimal drawings. Maybe it does for
%%% all sufficiently large $n$? Silvia, Bernardo, any thoughts?
%%% If Oswin verifies Kyncl's conjecture about $\le \le \le k$-edges for small values of $n$, this would be the place to mention it. If this does not
%%%  get verified --- well, of course we just won't mention anything.

\end{document}